\documentclass[a4paper,12pt]{amsart} 

\usepackage{amssymb,amsmath}
\usepackage{pb-diagram}
\usepackage{pb-xy}
\usepackage[cmtip,arrow]{xy}

\newtheorem{thm}{Theorem}[section] 
\newtheorem{prop}[thm]{Proposition}
\newtheorem{cor}[thm]{Corollary}

\theoremstyle{definition} 
\newtheorem{defn}{Definition} 
\newtheorem{rem}[thm]{Remark} 
\newtheorem{exa}{Example}

\numberwithin{equation}{section}

\begin{document}

\title{Frame bundle approach to generalized minimal submanifolds}

\author[K. Niedzia\l omski]{Kamil Niedzia\l omski}
\address{
Department of Mathematics and Computer Science \endgraf
University of \L\'{o}d\'{z} \endgraf
ul. Banacha 22, 90-238 \L\'{o}d\'{z} \endgraf
Poland\endgraf
}
\email{kamiln@math.uni.lodz.pl}

\begin{abstract}
We extend the notion of $r$--minimality of a submanifold in arbitrary codimension to $u$--minimality for a multi--index $u\in\mathbb{N}^q$, where $q$ is the codimension. This approach is based on the analysis on the frame bundle of orthonormal frames of the normal bundle to a submanifold and vector bundles associated with this bundle. The notion of $u$--minimality comes from the variation of $\sigma_u$--symmetric function obtained from the family of shape operators corresponding to all possible bases of the normal bundle. We obtain the variation field, which gives alternative definition of $u$--minimality. Finally, we give some examples of $u$--minimal submanifolds for some choices of $u$.
\end{abstract}


\maketitle

\section{Introduction}

The notion of minimality of a submanifold has long history. It has been considered by many authors within many contexts. Minimality can be defined by vanishing of the mean curvature -- the trace of the second fundamental form. The other possible extrinsic conditions we may impose on a submanifold are: total geodesicity, when second fundamental form vanishes, and total umbilicity, when the second fundamental form is proportional to the mean curvature vector.

While studying hypersurfaces, i.e. codimension one submanifolds, we may consider more invariants, which come from eigenvalues of the second fundamental form. We define the $r$--th curvature as a $r$--th symmetric function of the principal curvatures. This has been extensively considered since Reilly's work \cite{Rei}. Let us briethly recall his approach. Let $L$ be a hypersurface in a Riemannian manifold $(M,g)$ with the Levi--Civita connection $\nabla$. Denote by $B$ the second fundamental form and by $A=A^N$ the shape operator corresponding to a choice of normal vector $N$, 
\begin{equation*}
B(X,Y)=(\nabla_X Y)^{\bot},\quad g(A(X),Y)=g(B(X,Y),N),\quad X,Y\in TL.
\end{equation*}  
The symmetric function $S_r$ of $A$ is defined by the characteristic polynomial
\begin{equation*}
\chi_A(t)=\det (I+tA)=\sum_{r=0}^n S_r t^r,
\end{equation*}
where $n=\dim L$. It is a remarkable observation that the variation of $S_r$ is described by the Newton transformation $T_r=T_r(A)$. Namely, let
\begin{equation*}
T_r=\sum_{j=0}^r (-1)^j S_j A^{r-j},\quad r=0,1,\ldots,n.
\end{equation*}
Then, we have
\begin{equation}\label{eq:introvarsigmar}
\frac{d}{dt}S_r(t)={\rm tr}\left(\frac{d}{dt}A(t)\cdot T_{r-1}\right),
\end{equation}
where $S_r(t)$ is a one parameter family of symmetric functions corresponding to a family of operators $A(t)$, which satisfies $A(0)=A$ \cite{Rei}. With the use of the above formula Reilly obtained the variation of the integral of $S_r$ in the case of $M$ being of constant sectional curvature. The Euler--Lagrange equation defines the notion of $r$--minimality. Studies on submanifolds satisfying certain conditions involving $r$--th mean curvature has been recently very fruitful \cite{AC,ALM,BC,BS,CR}. Analogous considerations has been  also led in the case of foliations \cite{BC,AW1}. 

There has been several attempts to generalize this approach to submanifolds of arbitrary codimension (see \cite{Rei2,CL,CL2}) and to the case of arbitrary foliations \cite{BN,AW2}. The problems which appear are the following. First of all, the normal bundle to a submanifold is not trivial in a sense that the covariant derivative $\nabla^{\bot}$ does not annihilate unit vector fields. Secondly, there is no canonical choice of the orthonormal basis in the normal bundle. The second problem has been overcome (see \cite{Rei2,Gro}) by introducing transformations which depend only on the points on the manifold,
\begin{equation*}
(T_r)^i_j=\frac{1}{r!}\sum_{i_1,\ldots,i_r;j_1,\ldots,j_r}\delta^{i_1\ldots,i_r i}_{j_1\ldots, j_r j}g(B_{i_1j_1},B_{i_2j_2})\ldots g(B_{i_{r-1}j_{r-1}},B_{i_rj_r})
\end{equation*}  
for $r$ even, and
\begin{equation*}
(T_r^{\alpha})^i_j=\sum_{i_1,\ldots,i_r;j_1,\ldots,j_r}\delta^{i_1\ldots,i_r i}_{j_1\ldots, j_r j}g(B_{i_1j_1},B_{i_2j_2})\ldots g(B_{i_{r-2}j_{r-2}},B_{i_{r-1}j_{r-1}})(A^{\alpha})^i_j
\end{equation*}
for $r$ odd, where $\delta^{i_1,\ldots,i_r}_{j_1,\ldots,j_r}$ is the generalized Kronecker symbol and $B_{ij}=B(e_i,e_j)$ for some choice of orthonormal basis $(e_i)$. Then
\begin{align*}
S_r &=\frac{1}{r}\sum_{i,j,\alpha}(T_{r-1}^{\alpha})^i_j(A^{\alpha})^j_i && \textrm{for $r$ even},\\
{\bf S}_r &=\frac{1}{r}\sum_{i,j,\alpha}(T_{r-1})^i_j(A^{\alpha})^j_ie_{\alpha} && \textrm{for $r$ odd},
\end{align*}
denote the $r$--th symmetric function (vector field) of curvatures. In \cite{CL} the authors studied $r$--th minimality and stability with respect to these generalized transformations. The stability conditions led to non--existence results of stable minimal submanifolds on spheres \cite{Sim,CL}. 

In this article, we define the notion of minimality associated with the generalized Newton transformation introduced by Andrzejewski, Koz\l owski and the author \cite{AKN}. The idea of comes from then definition of symmetric functions associated with a system ${\bf A}=(A_1,\ldots,A_q)$ of matrices (endomorphisms),
\begin{equation*}
\chi_{\bf A}({\bf t})=\sum_{u}\sigma_u {\bf t}^u,
\end{equation*}
where $u=(u_1,\ldots,u_q)\in\mathbb{N}^q$ is a multi--index, ${\bf t}=(t_1,\ldots,t_q)$ and ${\bf t}^u=t_1^{u_1}\ldots, t_q^{u_q}$. Generalized Newton transformation $T_u=T_u({\bf A})$ depends on the multi--index $u$ and a system of endomorphisms ${\bf A}$ and its recursive definition is the following
\begin{equation*}
T_{\bf 0}=I,\quad T_u=\sigma_u I-\sum_{\alpha} A_{\alpha}T_{\alpha_{\flat}(u)},
\end{equation*} 
where $\alpha_{\flat}(u)$ is a multi--index obtained from $u$ by subtracting $1$ in the $\alpha$--th coordinate. Moreover, it satisfies analogue of the condition \eqref{eq:introvarsigmar}, namely,
\begin{equation*}
\frac{d}{dt}\sigma_u(t)=\sum_{\alpha} {\rm tr}\left(\frac{d}{dt}A_{\alpha}(t)\cdot T_{\alpha_{\flat}(u)}\right).
\end{equation*}

With the use of above characterization of $T_u$ we study the variation of $\sigma_u$. In this setting, $\sigma_u$ are symmetric functions of the family of shape operators $(A^{e_1},\ldots,A^{e_q})$ corresponding to the choice of the orthonormal basis $(e_{\alpha})$ in the normal bundle to $L\subset M$. Thus $\sigma_u$ are functions on the principal bundle $P$ of all normal orthonormal frames over the submanifold. The integrals of $\sigma_u$ over $P$ with respect to normalized measure are called total extrinsic curvatures and denoted by $\hat{\sigma_u}$. 

Notice, that in the case of distributions and foliations different types of curvatures has been recently considered \cite{RW,Rov}, mainly, in order obtain integral formulas. The ones used in this context have been introduced by Andrzejewski, Koz\l owski and the author \cite{AKN} for distributions to obtain different types of integral formulas related to extrinsic geometry. 

In this article, we derive the formula for the variation $\frac{d}{dt}\hat{\sigma}_u$. We show that at $t=0$
\begin{equation*}
\frac{d}{dt}\hat{\sigma}_u =\int_L g(R_u+W_u-S_u,V)\,d{\rm vol}_L,
\end{equation*} 
where $V$ is the variation field and $R_u, S_u, W_u$ are integrals over the fibers of the following sections of the normal bundle to $L$ over $P$,
\begin{align*}
&{\bf R}_u=-\sum_{\alpha,\beta}{\rm tr}((R(e_{\alpha},\cdot)e_{\beta})^{\top} T_{\alpha_{\flat}(u)})e_{\beta}, &&{\bf S}_u=\sum_{\alpha}{\rm tr}(A_{\alpha}T_u)e_{\alpha},\\
&{\bf W}_u=\sum_{\alpha}{\rm div}({\rm div} T_{\alpha_{\flat}(u)})e_{\alpha},
\end{align*}
respectively, where $X_{\alpha}$ denotes the $\alpha$ coefficient of $X$ with respect to $(e_{\alpha})\in P$. This condition can be interpreted the case of $M$ being a space form with sectional curvature $c$. Then divergences of $T_{\alpha_{\flat}(u)}$ vanish and ${\bf R}_u$ reduces to $c(n-|u|+1){\bf H}_u$, where $n=\dim L$ and
\begin{equation*}
{\bf H_u}=\sum_{\alpha}\sigma_{\alpha_{\flat}(u)}e_{\alpha}.
\end{equation*} 
Hence, in this case, $u$--minimality is equavalent to the condition (see Theorem \ref{thm:uminimality})
\begin{equation*}
c(n+1-|u|)H_u=S_u.
\end{equation*}
These results generalize the codimension one results by Reilly \cite{Rei} and, in a sense, the results of Cao and Li \cite{CL}. This is due to the fact that $T_u$ in codimension one reduces to the classical Newton transformation $T_r$, whereas transformations $T_r$ in arbitrary codimension used in \cite{CL} are obtained from $T_u$ for multi--indices of length $|u|=r$ (see \cite{AKN}). In the end, we give some examples of $u$--minimal submanifolds. Among others, we show that ${\bf 0}$--minimality is equivalent to classical notion of minimality.  

Throughout the paper, we use the following index convention:
\begin{equation*}
i,j,k=1,2,\ldots,\dim L;\quad \alpha,\beta,\gamma=q+1,q+2,\ldots,\dim M.
\end{equation*}

\subsection*{Acknowledgment}
The author wishes to thank Krzysztof Andrzejewski, Wojciech Koz\l owski and Pawe{\l} Walczak for fruitful conversations and discussions on this topic. Last but not least, the author wishes to thank anonymous referee for critical look at the previous version of the article, for finding many mistakes and for many valuable comments.

\section{Generalized Newton transformation and its basic properties}

In \cite{AKN} the authors introduced the notion of the generalized Newton transformation in order to study geometry of foliations of codimension higher than one. This transformation generalizes the classical Newton transformation to the case of finite family of operators. Let us recall this notion and state its properties.  

Fix a positive integer $q$ and let $u\in\mathbb{N}^q$ ($\mathbb{N}$ denotes here the set of all non--negative integers). We define the length of $u$ by $|u|=u_1+u_2+\ldots,+u_q$. Let ${\bf A}=(A_1,A_2,\ldots,A_q)$ be a family of square $n$ by $n$ operators (matrices) on some vector space $V$. Then we can consider the characteristic polynomial $\chi_{\bf A}$ of the form
\begin{equation*}
\chi_{\bf A}({\bf t})=\det (I+{\bf t}{\bf A}),
\end{equation*}
where ${\bf t}=(t_1,\ldots,t_q)$ and ${\bf t}{\bf A}=t_1A_1+t_2A_2+\ldots+t_qA_q$. Clearly,
\begin{equation*}
\chi_{\bf A}({\bf t})=\sum_{|u|\leq n}\sigma_u({\bf A}){\bf t}^u,
\end{equation*}
for some constants $\sigma_u=\sigma_u({\bf A})$, where ${\bf t}^u=t_1^{u_1}t_2^{u_2}\ldots t_q^{u_q}$ for $u=(u_1,\ldots,u_q)$. We call $\sigma_u$ the symmetric functions associated with the system ${\bf A}$.

\begin{defn}
A system $T_u=T_u({\bf A})$, $|u|\leq n$, of endomorphisms is called the {\it generalized Newton transformation} if the following condition holds: Let ${\bf A}(t)$ be a one parameter family of operators such that ${\bf A}(0)={\bf A}$ and let $\sigma_u(t)$ be the corresponding symmetric functions. Then
\begin{equation}\label{eq:defTu}
\frac{d}{dt}\sigma_u(t)_{t=0}=\sum_{\alpha}{\rm tr}\left(\frac{d}{dt}A_{\alpha}(t)_{t=0}\cdot T_{\alpha_{\flat}(u)}\right),
\end{equation} 
where $\alpha_{\flat}(u)$ denotes the multi--index obtained from $u$ by subtracting $1$ in the $\alpha$--th coordinate.
\end{defn}
One can show that $(T_u)_{|u|\leq n}$ exists and is unique \cite{AKN}. Moreover, it can be characterized by the following recursive correspondence
\begin{align}
T_0 &=I, && 0=(0,0,\ldots,0), \notag\\
T_u &=\sigma_u I-\sum_{\alpha} A_{\alpha}T_{\alpha_{\flat}(u)}
\label{eq:recurencepropTu}\\
&=\sigma_u I-\sum_{\alpha} T_{\alpha_{\flat}(u)}A_{\alpha},
&& |u|\geq 1.\notag
\end{align}  

The generalized Newton transformation has the following useful properties \cite{AKN}.
\begin{prop}\label{prop:propTu}
Let ${\bf A}=(A_1,\ldots,A_q)$ be a family of operators on a vector space $V$. Let $(T_u)$ be the generalized Newton transformation associated with ${\bf A}$ and let $\sigma_u$ be the corresponding symmetric functions. Then the following relations hold
\begin{align*}
|u|\sigma_u &=\sum_{\alpha}{\rm tr}(A_{\alpha}T_{\alpha_{\flat}(u)}),\\
{\rm tr}T_u &=(n-|u|)\sigma_u.
\end{align*} 
\end{prop}
 
One can obtain the direct formula for $T_u$ and for $\sigma_u$ (see \cite{AKN}, eq. (3) and (10)).

\section{Frame bundle approach to submanifold geometry}\label{sec:Submgeom}

\subsection{Bundle constructions}
Let $(M,g)$ be a Riemannian manifold, $\varphi:L\mapsto M$ a codimension $q$ submanifold, such that $\varphi$ is an isometric immersion. Thus, we identify $L$ with its image $\varphi(L)\subset M$. Let $\nabla$ be the Levi-Civita connection of $g$. Consider the pull--back bundle $TM|L$ over $L$, which is just a tangent bundle of $M$ restricted to $L$. This bundle splits into two bundles $TL$ and $T^{\bot}L$ with respect to the decomposition
\begin{equation}\label{eq:TMsplitting}
T_M=T_pL\oplus T_p^{\bot}L,\quad p\in L.
\end{equation}
Denote by $\nabla^{\top}$ and $\nabla^{\bot}$ the induced connections in these bundles, respectively. 

Let $P=O(T^{\bot}L)$ (or $P=SO(T^{\bot}L)$ assuming $L$ is transversally orientable) be a bundle of (oriented) orthonormal frames of $T^{\bot}L$. Alternatively, $\Gamma(P)$ consists of all sections, which assign an orthonormal basis $(p,e)=(e_{\alpha})_{\alpha=1,\ldots,q}$, where $e_{\alpha}\in T^{\bot}_pL$, to any point $p\in L$. 

Moreover, let $\pi_P^{-1}TM|L$ be a vector bundle over $P$ with the fiber
\begin{equation*}
(\pi_P^{-1}TM|L)_{(p,e)}=T_pM.
\end{equation*}
This bundle splits into direct (orthogonal) sum of bundles $E$ and $E'$ with respect to the decomposition \eqref{eq:TMsplitting}. In other words,
\begin{equation*}
E_{(p,e)}=T_pL,\quad E'_{(p,e)}=T^{\bot}_pL.
\end{equation*}
Notice that both bundles $E$ and $E'$ are pull--back bundles $\pi^{-1}_PTL$ and $\pi_P^{-1}T^{\bot}L$, respectively. The connections $\nabla^{\top}$ and $\nabla^{\bot}$ induce pull--back connections $\nabla^E$ and $\nabla^{E'}$ in bundles $E$ and $E'$, respectively, i.e.
\begin{equation*}
\nabla^E_X (Y\circ \pi_P)=\nabla^{\top}_{\pi_{P\ast}X}Y,\quad
\nabla^{E'}_X (Y\circ \pi_P)=\nabla^{\bot}_{\pi_{P\ast}X}Z
\end{equation*}  
where $X\in TP$, $Y\in\Gamma(TL)$, $V\in\Gamma(T^{\bot}L)$.

\subsection{Frame bundle calculus}\label{subsec:bundlecalculus}
In this subsection we briefly review some basic facts about differentiation and integration on the frame and associated vector bundles, which we will need later. For detailed (general) description see \cite[Appendix]{AKN}. Let us keep the notation from the previous subsection.

Connection $\nabla^{\bot}$ on $L$ induces horizontal distribution $\mathcal{H}$ on $P$. Then, each vector $X\in T_pL$ has a unique horizontal lift $X^h_{(p,e)}$ to $\mathcal{H}_{(p,e)}$. Let $e$ be a local section of $P$ such that $e_{p\ast} X=X^h_{(p,e)}$. Then $X^h_{(p,e)}f=X(f\circ e)(p)$ for any $f\in C^{\infty}(P)$. Moreover, for such $e$ we have at $p$
\begin{equation*}
(\nabla^E_{X^h}Y)\circ e=\nabla^{\top}_X(Y\circ e),\quad (\nabla^{E'}_{X^h}Y)\circ e=\nabla^{\bot}_X(Y\circ e)
\end{equation*}

Denoting by $\lambda_G$ the (normalized) Haar measure on $G$, we may define integration on the fiber $P_p$ as follows
\begin{equation*}
\int_{P_p} f\,d{\rm vol}_{P_p}=\int_G f((p,e_0)\cdot g)\, d\lambda_G.
\end{equation*}
One can see that the integral on the right hand side does not depend on the choise of the reference basis $(p,e_0)$. Denote the above integral by $\hat{f}(p)$. Then integral over $P$ is just the integral of $\hat{f}$ over $L$. The following useful formula holds
\begin{equation*}
\widehat{X^hf}=X\hat{f},\quad X\in\Gamma(TL),\quad f\in C^{\infty}(P).
\end{equation*}
Moreover, notice that if $Y\in \Gamma(E)$ (or $Y\in\Gamma(E')$, respectively), then we may average $Y$ in the fibers to obtain a vector field $\hat{Y}\in \Gamma(TL)$ (or $\hat{Y}\in\Gamma(T^{\bot}L)$, respectively) by integrating coordinates of $Y$. It can be shown that such integral is well defined, i.e. independent of the choice of point--wise basis.

Finally, let us define the notion of gradient, hessian and divergence with respect to $E$ and restricted to the horizontal distribution. Denote by ${\rm div}_E T_u$ the divergence of $T_u$, which a section of the bundle $E$ of the form
\begin{equation*}
{\rm div}_E T_u=\sum_i (\nabla^E_{e_i^h} T_u)(e_i\circ\pi_P),
\end{equation*}
where $(e_i)$ is any orthonormal basis on $L$. Then the following recurrence formula holds \cite{AKN}
\begin{equation}\label{eq:divTu}
{\rm div}_E T_u=-\sum_{\alpha} A_{\alpha}({\rm div}_E T_{\alpha_{\flat}(u)})+\sum_{\alpha,i}R(e_{\alpha},T_{\alpha_{\flat}(u)}e_i)e_i)^{\top},\quad {\rm div}_E T_0=0.
\end{equation}

Moreover, ${\rm grad} f\in \Gamma(E)$, ${\rm Hess}\,f\in\Gamma(E^{\ast}\otimes E^{\ast})$ and ${\rm div}_E(W)\in C^{\infty}(P)$ are defined as follows
\begin{align*}
& g({\rm grad}_E\,f,X)=X^hf, &&f\in C^{\infty}(P),\quad X\in TL,\\
& {\rm Hess}\,f(X,Y)=g(\nabla^E_{X^h}{\rm grad}_E\,f,Y), && f\in C^{\infty}(P),\quad X,Y\in TL,\\
& {\rm div}_E(W)=\sum_i g(\nabla^E_{e_i^h}W,e_i\circ\pi_P), &&W\in \Gamma(E),
\end{align*}
where $(e_i)$ is an orthonormal basis in $TL$. Notice that ${\rm Hess}\,f$ is not, in general symmetric in $X,Y$. The skew--symmetric part $R(X,Y)^{\ast}f$, where upper index $\ast$ denotes the fundamental vertical vector field,  vanishes if $f$ is constant on the fibers, hence, if $f=f_L\circ\pi_P$ for some $f_L\in C^{\infty}(L)$.

\subsection{Generalized extrinsic curvatures}
Let $N\in T^{\bot}_pL$. Then $N$ induces the shape operator $A^N:T_pL\to T_pL$ by the formula
\begin{equation*}
A^N(X)=-\left(\nabla_XN\right)^{\top} ,\quad X\in T_pL.
\end{equation*}
Varying $N\in\Gamma(E')$ we have $A^N\in\Gamma({\rm End}(E))$. For an element $(p,e)$ of $P$ we may consider a family of shape operators
\begin{equation*}
{\bf A}(p,e)=(A^{e_1},A^{e_2},\ldots,A^{e_q}),\quad (p,e)=(e_1,\ldots,e_q).
\end{equation*}
We will write $A_{\alpha}$ instead of $A^{e_{\alpha}}$. To each such family we can associate the symmetric functions
\begin{equation*}
\sigma_u(p,e)=\sigma_u({\bf A}(p,e)),
\end{equation*} 
where $u=(u_1,u_2,\ldots,u_q)\in\mathbb{N}^q$ is a fixed multi--index. Thus $\sigma_u\in C^{\infty}(P)$. Moreover, the generalized Newton transformation $T_u$ associated with ${\bf A}$ is a section of ${\rm End}(E)$.

\begin{defn} 
Function $\sigma_u$ induces a smooth function $\hat{\sigma}_u$ on $L$,
\begin{equation*}
\hat{\sigma}_u(p)=\int_{P_p}\sigma_u\,d{\rm vol}_{P_p},\quad p\in L.
\end{equation*}
We call $\hat{\sigma}_u$ the {\it generalized extrinsic curvature} of $L$ or $u$--{\it extrinsic curvature}. 
\end{defn}

\begin{rem}
For some choices of $u$ the generalized extrinsic curvatures vanish, but there are many multi--indices $u$ for which $\hat{\sigma_u}$ is non--zero and gives information on the submanifold.

For indices $1\leq\alpha_1<\ldots<\alpha_k\leq q$, consider a map $F_{\alpha_1,\ldots,\alpha_k}:P\to P$ which maps vector $e_{\alpha_i}$ to $-e_{\alpha_i}$. Then
\begin{equation*}
\sigma_u(F_{\alpha_1,\ldots,\alpha_k}(p,e))=(-1)^{u_{\alpha_1}+\ldots+u_{\alpha_k}}\sigma_u(p,e).
\end{equation*}
Therefore the following two conditions hold (compare \cite{AKN}):
\begin{itemize}
\item if $G=O(q)$ and at least one of indices $u_1,\ldots,u_q$ is odd, then $\hat{\sigma}_u=0$,
\item if $G=SO(q)$ and there is one index odd and one even, then $\hat{\sigma}_u=0$.
\end{itemize}
It follows, that if all indices $u_1,\ldots,u_q$ are even then $\sigma_u$ is not, in general, equal to zero. Thus the first nontrivial case is $u=\alpha^{\sharp}\alpha^{\sharp}(0)$ (compare Example \ref{ex:4}).

Moreover, consider the map $H_{\tau}:P\to P$, where $\tau\in\mathfrak{S}_q$ is a permutation of $q$ elements, which permutes the elements in the basis,
\begin{equation*}
H_{\tau}(p,e_1,\ldots,e_q)=(p,e_{\tau(1)},\ldots,e_{\tau(q)}).
\end{equation*} 
Then, looking at the both sides of the characteristic polynomial of the system ${\bf A}=(A_1,\ldots,A_q)$ we conclude that 
\begin{equation*}
\sigma_u(H_{\tau}(p,e))=\sigma_{u\circ \tau}(p,e).
\end{equation*}
Integrating both sides over $P$ we obtain the following relations:
\begin{itemize}
\item if $G=O(q)$, then $\hat{\sigma}_u=\hat{\sigma}_{u\circ\tau}$ for all permutations $\tau\in\mathfrak{S}_q$,
\item if $G=SO(q)$, then $\hat{\sigma}_u=\hat{\sigma}_{u\circ\tau}$ for all even permutations $\tau\in\mathfrak{S}^{\rm even}_q$.
\end{itemize}
Therefore, instead of $\hat{\sigma}_u$ we may consider quantities
\begin{align*}
\hat{\sigma}_u &=\frac{1}{q!}\sum_{\tau\in\mathfrak{S}_q}\hat{\sigma}_{u\circ\tau} && \textrm{for $G=O(q)$},\\
\hat{\sigma}_u &=\frac{1}{(q-1)!}\sum_{\tau\in\mathfrak{S}^{\rm even}_q}\hat{\sigma}_{u\circ\tau} && \textrm{for $G=SO(q)$}.
\end{align*}
Notice, that in particular $\alpha^{\sharp}\alpha^{\sharp}(0)$--minimality is equivalent to being a critical point of the variation of $\frac{1}{q}\left(\hat{\sigma}_{(2,0,\ldots,0,0)}+\ldots+\hat{\sigma}_{(0,0,\ldots,0,2)}\right)$ (compare Example \ref{ex:4} in the last section).
\end{rem}

Let us define for a fixed multi--index $u\in\mathbb{N}^q$ four sections ${\bf R}_u,{\bf H}_u,{\bf S}_u, {\bf W}_u\in\Gamma(E')$. Namely, we put
\begin{equation}\label{eq:boldHuandSu}
\begin{split}
{\bf R}_u(p,e) &=-\sum_{\alpha,\beta} {\rm tr}(R_{\alpha\beta} T_{\alpha_{\flat}(u)})e_{\beta}, \\
{\bf H}_u(p,e) &=\sum_{\alpha}\sigma_{\alpha_{\flat}(u)}e_{\alpha},\\
{\bf S}_u(p,e) &=\sum_{\alpha}{\rm tr}(A_{\alpha}T_u)e_{\alpha}, \\
{\bf W}_u(p,e) &=\sum_{\alpha}{\rm div}_E({\rm div}_E T_{\alpha_{\flat}(u)})e_{\alpha}, 
\end{split}
\end{equation} 
where $R_{\alpha\beta}\in \Gamma(E^{\ast}\otimes E^{\ast})$ is given by
\begin{equation}\label{eq:Ralphabeta}
R_{\alpha,\beta}(p,e)(X,Y)=g(R(e_{\alpha},X)e_{\beta},Y),\quad X,Y\in T_pL.
\end{equation}
These sections induce sections $R_u$, $H_u$ and $S_u$ of $T^{\bot}L$, respectively, by integrating over the fibers. Let us describe above quantities in the codimension one case. We have
\begin{align*}
R_r &=-{\rm tr}(R_{NN}T_{r-1})N=\sum_i (R(N,T_{r-1}e_i)e_i)^{\bot},\\
H_r &=\sigma_{r-1}N,\\
S_r &={\rm tr}(AT_r)N=(r+1)\sigma_r N,\\
W_r &={\rm div}({\rm div} T_{r-1})N.
\end{align*}

\subsection{Differential operators}
In this subsection we generalize some of the results obtained by Rosenberg \cite{Ros} concerning differential operators induced by the Newton transformation. Let us begin with the fact about differentiation of symmetric functions $\sigma_u$ in the horizontal direction \cite{AKN}. Denote by $\mathcal{H}$ the horizontal distribution in $P$ induced by the connection $\nabla^{\bot}$.
\begin{prop}\label{prop:differentialsigmau}
The following relations hold
\begin{equation}\label{eq:differentialsigmau}
\begin{split}
X^h\sigma_u &=\sum_{\alpha}{\rm tr}\left(\left(\nabla^E_{X^h}A_{\alpha}\right)\cdot T_{\alpha_{\flat}(u)}\right)\\
&=\sum_{\alpha}{\rm tr}\left(\left(\nabla^E_{X^h}A\right)_{\alpha}\cdot T_{\alpha_{\flat}(u)}\right),\quad X\in TL.
\end{split}
\end{equation} 
\end{prop}

Recall that the covariant derivative $(\nabla^E_{X^h}A)_{\alpha}$ is defined as
\begin{equation*}
(\nabla^E_{X^h}A)_{\alpha}Y=\nabla^E_{X^h}A_{\alpha}Y-A_{\alpha}(\nabla^E_{X^h}Y)-A^{\nabla^{E'}_{X^h}e_{\alpha}}Y,\quad Y\in \Gamma(E).
\end{equation*}
\begin{proof}
Fix a basis $(p,e)\in P$ and a vector $X\in T_pL$. Let $\gamma$ be a curve on $L$ such that $\gamma(0)=p$ and $\gamma '(0)=X$. Consider the horizontal lift $\gamma^h$ to $(p,e)$, i.e., $\gamma^h(0)=(p,e)$ and $(\gamma^h)'(0)=X^h_{(p,e)}$. Denote by $\tau_t$ the parallel displacement with respect to $\gamma$ from $p$ to $\gamma(t)$. Put
\begin{equation*}
\tilde{{\bf A}}(t)=\tau_t^{-1}\circ{\bf A}(\gamma^h(t))\circ\tau_t,\quad \textrm{and}\quad
\tilde{\sigma}_u(t)=\sigma_u(\tilde{\bf A}(t)).
\end{equation*}
Then $\tilde{\bf A}=(\tilde{A}_1,\ldots,\tilde{A}_q)$ is a one--parameter family of $q$ endomorphisms of $T_pL$. Notice that $\tilde{{\bf A}}(0)={\bf A}(p,e)$. Thus, by a definition of generalized Newton transformation, we get
\begin{equation*}
\frac{d}{dt}\tilde{\sigma}(t)_{t=0}=\sum_{\alpha}{\rm tr}\left( \frac{d}{dt}\tilde{A}_{\alpha}(t)_{t=0}\cdot T_{\alpha_{\flat}(u)}(p,e)\right).
\end{equation*}
Moreover,
\begin{equation*}
\frac{d}{dt}\tilde{\sigma}_u(t)_{t=0}=\frac{d}{dt}\sigma_u(\gamma^h(t))_{t=0}=(X^h\sigma_u)(p,e)
\end{equation*}
and for $Y\in T_pL$
\begin{align*}
\frac{d}{dt}\tilde{A}_{\alpha}(t)_{t=0}Y &=\frac{d}{dt}\left(\tau_t^{-1}A^{e_{\alpha}(\gamma^h(t))}\tau_t Y\right)_{t=0}=\nabla^{\top}_X A_{\alpha}Y-A^{\nabla^{\bot}_Xe_{\alpha}}Y-A_{\alpha}(\nabla^{\top}_XY)\\
&=(\nabla^{\top}_XA)_{\alpha}Y.
\end{align*}
Since we consider bases $e\in P$ along $\gamma^h$, is follows that each $e_{\alpha}$ is parallel along $X$, $\nabla^{\bot}_Xe_{\alpha}=0$. Now \eqref{eq:differentialsigmau} follows from relation $\nabla^{\top}_XA=(\nabla^E_{X^h}A)\circ e$ (see subsection \ref{subsec:bundlecalculus}).
\end{proof}

For any section $W\in\Gamma(E')$ denote by $W_{\alpha}$ a (smooth) function on $P$ of the form
\begin{equation*}
W_{\alpha}(p,e)=g(W,e_{\alpha})_p,\quad (p,e)\in P.
\end{equation*}
Notice that any vector field $X\in\Gamma(TL)$ induces a vector field $X\circ\pi_P\in\Gamma(E)$ and, hence, a function $(X\circ\pi_P)_{\alpha}$ denoted shortly by $X_{\alpha}$. Notice, moreover, that
\begin{equation}\label{eq:Hessalpha}
{\rm Hess}\, W_{\alpha}(X,Y)=((\nabla^{\bot})^2_{X,Y} W)_{\alpha},\quad W\in\Gamma(T^{\bot}L).
\end{equation}
In deed, fix a basis $(p,e)\in P$ and extend it locally to a parallel section $e\in\Gamma(P)$ at $(p,e)$. Then at $(p,e)$
\begin{align*}
{\rm Hess}\, W_{\alpha}(X,Y) &=g(\nabla^{\top}_X({\rm grad}_E\,W_{\alpha}\circ e),Y)\\
&=X g({\rm grad}_E\,W_{\alpha}\circ e,Y)-g({\rm grad}_E\,W_{\alpha}\circ e,\nabla^{\top}_XY)\\
&=XYg(W,e_{\alpha})-(\nabla^{\top}_XY)g(W,e_{\alpha})\\
&=g(\nabla^{\bot}_X(\nabla^{\bot}_YW),e_{\alpha})-g(\nabla^{\bot}_{\nabla^{\top}_XY}W,e_{\alpha})\\
&=g((\nabla^{,\bot})^2_{X,Y}W,e_{\alpha}).
\end{align*}

\begin{rem}
Notice that above relation can be proved more abstractly as follows. The bundle $T^{\bot}L$ is an associated bundle to the frame bundle $P$ with the fiber $\mathbb{R}^q$. There is one--to--one correspondence of the sections of this bundle with equivariant functions on $P$ with the values in the fiber $\mathbb{R}^q$. The equivariant function corresponding to the vector field $W$ equals $f_W(p,e)=(W_1,\ldots,W_q)=e^{-1}W$. Thus the $\alpha$--coordinate of $f_W$ equals $W_{\alpha}$. Since, $X^hf_W$ corresponds to the vector field $\nabla^{\bot}_XW$ it follows that $X^hW_{\alpha}$ corresponds to $(\nabla^{\bot}_XW)_{\alpha}$. In particular, applying this correspondence twice, we obtain the above equality.  
\end{rem}

\begin{defn}
Fix a multi--index $u\in\mathbb{N}^q$ and consider differential operators $L^{\ast}_u:\Gamma(E')\to C^{\infty}(P)$ and $L_u: C^{\infty}(P)\to\Gamma(E')$
\begin{align}
L^{\ast}_u(W) &=\sum_{\alpha}{\rm tr}({\rm Hess}\,W_{\alpha}\cdot T_{\alpha_{\flat}(u)}),\label{eq:diffopast}\\
L_u(f) &=\sum_{\alpha}{\rm tr}({\rm Hess}\,f\cdot T_{\alpha_{\flat}(u)})e_{\alpha}.\label{eq:diffop}
\end{align}
\end{defn}

Let us derive the formula for the value of $L^{\ast}_u$ on sections of the form $V\circ\pi_P$, where $V\in \Gamma(TL)$ and value of $L_u$ on functions of the form $f\circ\pi_P$, where $f\in C^{\infty}(L)$ and relations between them. We need some preliminary results being the adaptations of the results of Rosenberg \cite{Ros} to our setting.

Recall the Codazzi formula
\begin{equation}\label{eq:Codazzisingle}
(R(X,Y)N)^{\top}=(\nabla^{\top}_YA)^NX-(\nabla^{\top}_XA)^NY,\quad X,Y\in TL,\quad N\in  T^{\bot}L,
\end{equation}
where
\begin{equation*}
(\nabla^{\top}_XA)^NY=\nabla^{\top}_X(A^NY)-A^N(\nabla^{\top}_XY)-A^{\nabla^{\bot}_XN}Y.
\end{equation*}
We are ready to state and prove the main results of this section.

\begin{prop}\label{prop:Rosresult}
The following condition holds
\begin{equation*}
{\rm div}_E(T_u Y)=\sum_i g(\nabla^E_{(T_ue_i)^h}Y,e_i\circ\pi_P)+g({\rm div}_E T_u,Y),\quad Y\in\Gamma(E). 
\end{equation*}
\end{prop}
\begin{proof}
Notice first that above formula is tensorial with respect to $Y\in TL$. Consider a parallel section $e$ of $P$ at $p\in L$. Then, above formula, at $(p,e)$ takes the form
\begin{equation}\label{eq:Rosformula}
\sum_i g(\nabla^{\top}_{e_i}(T_u\circ e)Y,e_i)=\sum_i g(\nabla^{\top}_{(T_u\circ e)e_i}Y,e_i)+g({\rm div}_ET_u,Y)\circ e.
\end{equation}
We will prove \eqref{eq:Rosformula} by induction on $|u|$. For $u=0$, it is clear since $T_0={\rm Id}$ and ${\rm div}_E T_0=0$. Fix $u$ and assume \eqref{eq:Rosformula} holds for all $v$ such that $|v|=|u|-1$. We may assume that $Y$ is parallel at fixed point $p$ with respect to $\nabla^{\top}$. Then, by the recurrence definition of $T_u$, Proposition \ref{prop:differentialsigmau}, Codazzi formula \eqref{eq:Codazzisingle} and inductive assumption we have
\begin{align*}
\sum_i g(\nabla^{\top}_{e_i}(T_u\circ e)Y,e_i) &=\sum_i g(\nabla^{\top}_{e_i}((\sigma_u\circ e) Y-\sum_{\alpha}(T_{\alpha_{\flat}(u)}\circ e)(A_{\alpha}\circ e)Y,e_i)\\
&=Y(\sigma_u\circ e)-\sum_{\alpha,i}g(\nabla^{\top}_{e_i}(T_{\alpha_{\flat}(u)}\circ e)(A_{\alpha}\circ e)Y,e_i)\\
&=Y^h\sigma_u\circ e-\sum_{\alpha,i}g(\nabla^{\top}_{(T_{\alpha_{\flat}(u)}\circ e)e_i}(A_{\alpha}\circ e)Y,e_i)\\
&-\sum_{\alpha} g({\rm div}T_{\alpha_{\flat}(u)},A_{\alpha}Y)\circ e\\
&=Y^h\sigma_u\circ e-\sum_{\alpha,i,j}g((T_{\alpha_{\flat}(u)}\circ e)e_i,e_j)g(\nabla^{\top}_{e_j}(A_{\alpha}\circ e)Y,e_i)\\
&-\sum_{\alpha} g(A_{\alpha}({\rm div}T_{\alpha_{\flat}(u)}),Y)\circ e\\
&=Y^h\sigma_u\circ e-\sum_{\alpha,i,j}g((T_{\alpha_{\flat}(u)}\circ e)e_i,e_j)g((\nabla^{\top}_{e_j}(A_{\alpha}\circ e))Y,e_i)\\
&-\sum_{\alpha} g(A_{\alpha}({\rm div}T_{\alpha_{\flat}(u)}),Y)\circ e\\
&=Y^h\sigma_u\circ e-\sum_{\alpha,i,j}g((T_{\alpha_{\flat}(u)}\circ e)e_i,e_j)g((\nabla^{\top}_Y(A_{\alpha}\circ e))e_j,e_i)\\
&+\sum_{\alpha,i,j}g((T_{\alpha_{\flat}(u)}\circ e)e_i,e_j)g(R(e_j,Y)e_{\alpha},e_i)\\
&-\sum_{\alpha} g(A_{\alpha}({\rm div}T_{\alpha_{\flat}(u)}),Y)\circ e\\
&=g({\rm div}_ET_u,Y)\circ e,
\end{align*}
which proves \eqref{eq:Rosformula}.
\end{proof}

\begin{prop}\label{prop:diffoperators}
Fix a multi--index $u$ a function $f\in C^{\infty}(L)$ and a vector field $V\in \Gamma(T^{\bot}L)$. Assume $V$ vanishes in some neighborhood of $\partial L$. Then, the following relation between differential operators $L_u$ and $L^{\ast}_u$ holds
\begin{align*}
\int_P (f\circ\pi_P)L_u^{\ast}(V\circ\pi_P)\,d{\rm vol}_P &=\int_P g(L_u(f\circ\pi_P),V\circ\pi_P)\,d{\rm vol}_P\\
&+\sum_{\alpha}\int_P g({\rm div}_E T_{\alpha_{\flat}(u)},V_{\alpha}{\rm grad} f\circ\pi_P)\,d{\rm vol}_P\\
&-\sum_{\alpha}\int_P g({\rm div}_E T_{\alpha_{\flat}(u)},(f\circ\pi_P){\rm grad}_E\, V_{\alpha})\,d{\rm vol}_P.
\end{align*}
In particular, taking $f=1$ we get
\begin{equation}\label{eq:Luastintegral}
\int_P L^{\ast}_u(V\circ\pi_P)\,d{\rm vol}_P=\sum_{\alpha}\int_P V_{\alpha}{\rm div}_E({\rm div}_E T_{\alpha_{\flat}(u)})\,d{\rm vol}_P.
\end{equation} 
\end{prop}
\begin{proof}
First notice that
\begin{align*}
g(L_u(f\circ\pi_P),V\circ\pi_P) &=\sum_{\alpha} V_{\alpha}{\rm tr}({\rm Hess}\,(f\circ\pi_P)\cdot T_{\alpha_{\flat}(u)})\\
&=\sum_{\alpha}V_{\alpha}{\rm tr}(({\rm Hess}\,f) \circ\pi_P\cdot T_{\alpha_{\flat}(u)})\\
&=\sum_{\alpha} V_{\alpha} \sum_{i,j}g(\nabla^{\top}_{e_i}{\rm grad}\,f,e_j)g(T_{\alpha_{\flat}(u)}e_j,e_i)\\
&=\sum_{\alpha,j}V_{\alpha}g(\nabla^{\top}_{T_{\alpha_{\flat}(u)}e_j}{\rm grad}\,f,e_j).
\end{align*}
Hence, by Proposition \ref{prop:Rosresult} we get
\begin{align*}
g(L_u(f\circ\pi_P),V\circ\pi_P) &=\sum_{\alpha}V_{\alpha}{\rm div}_E(T_{\alpha_{\flat}(u)}{\rm grad}\,f)-\sum_{\alpha}V_{\alpha}g({\rm div}_E T_{\alpha_{\flat}(u)},{\rm grad}\,f)\\
&=\sum_{\alpha}{\rm div}(V_{\alpha}T_{\alpha_{\flat}(u)}{\rm grad}\,f)-(T_{\alpha_{\flat}(u)}{\rm grad}\,f)^hV_{\alpha}\\
&-\sum_{\alpha}V_{\alpha}g({\rm div}_E T_{\alpha_{\flat}(u)},{\rm grad}\,f).
\end{align*}
Computing the integral of both sides we get
\begin{multline}\label{eq:diffoperatirseq1}
\int_P g(L_u(f\circ\pi_P),V\circ\pi_P)\,d{\rm vol}_P\\
=-\sum_{\alpha}\int_P ((T_{\alpha_{\flat}(u)}{\rm grad}\,f)^hV_{\alpha}+V_{\alpha}g({\rm div}_E T_{\alpha_{\flat}(u)},{\rm grad}\,f))\,d{\rm vol}_P.
\end{multline}
On the other hand, again by Proposition \ref{prop:Rosresult},
\begin{align*}
(f\circ\pi_P)L^{\ast}_u(V\circ\pi_P) &=(f\circ\pi_P)\sum_{\alpha,i,j}g(\nabla^E_{e_i^h}{\rm grad}_E V_{\alpha},e_j)g(T_{\alpha_{\flat}(u)}e_j,e_i)\\
&=(f\circ\pi_P)\sum_{\alpha}{\rm div}_E(T_{\alpha_{\flat}(u)}{\rm grad}_E\,V_{\alpha})\\
&-(f\circ\pi_P)\sum_{\alpha}g({\rm div}_E T_{\alpha_{\flat}(u)},{\rm grad}\,V_{\alpha})\\
&=\sum_{\alpha}{\rm div}_E((f\circ\pi_P)T_{\alpha_{\flat}(u)}{\rm grad}_E\,V_{\alpha})-\sum_{\alpha}(T_{\alpha_{\flat}(u)}{\rm grad}_E\,V_{\alpha})f\\
&-(f\circ\pi_P)\sum_{\alpha}g({\rm div}_E T_{\alpha_{\flat}(u)},{\rm grad}\,V_{\alpha}).
\end{align*}
Computing the integral of both sides we get
\begin{multline}\label{eq:diffoperatorseq2}
\int_P (f\circ\pi_P)L^{\ast}_u(V\circ\pi_P)\,d{\rm vol}_P\\
=-\sum_{\alpha}\int_P ((T_{\alpha_{\flat}(u)}{\rm grad}_E\,V_{\alpha})f+(f\circ\pi_P)\sum_{\alpha}g({\rm div}_E T_{\alpha_{\flat}(u)},{\rm grad}\,V_{\alpha}))\,d{\rm vol}_P.
\end{multline}
Moreover, by the symmetry of generalized Newton transformation,
\begin{align*}
(T_{\alpha_{\flat}(u)}{\rm grad}_E\, V_{\alpha})f &=g(T_{\alpha_{\flat}(u)}{\rm grad}_E\,V_{\alpha},{\rm grad}\,f)\\
&=g({\rm grad}_E\,V_{\alpha},T_{\alpha_{\flat}(u)}{\rm grad}\,f)\\
&=(T_{\alpha_{\flat}(u)}{\rm grad}\,f)^hV_{\alpha}.
\end{align*}
This, together with \eqref{eq:diffoperatirseq1} and \eqref{eq:diffoperatorseq2} proves the first part of the proposition. The second part follows from the fact that ${\rm div}_E(fY)=f{\rm div}_E(Y)+Y^hf$ for $f\in C^{\infty}(L)$ and $Y\in\Gamma(E)$.
\end{proof}

\section{Generalized minimality}

\subsection{Definitions and basic results}
Let $(M,g)$ be a Riemannian manifold, $\varphi:L\to M$ an immersed submanifold (possibly with boundary $\partial L$). Adopt the notation from the previous sections. Fix a multi--index $u\in\mathbb{N}^q$. 

\begin{defn}
We say that $L$ is $u$--{\it minimal} if $L$ is a critical point of the variation of generalized extrinsic curvature $\hat{\sigma}_u$ of $L$. We assume that the boundary  $\partial L$ is strongly fixed, i.e., the variation field vanishes in some neighborhood of the boundary.
\end{defn}

Let us be more precise. Let $\Phi:L\times\mathbb{R}\to M$ be a variation of $L$, i.e., $\Phi(\cdot,0)=\varphi$. We will write $\Phi_t=\Phi(\cdot,t)$. We assume $\Phi_t$ is an isometric immersion for each $t$. Each submanifold $L_t$ induces generalized extrinsic curvature $\hat{\sigma}_u(t)$. Then, $L$ is $u$--minimal if $\frac{d}{dt}\hat{\sigma}_u(t)_{t=0}=0$ for any variation $\Phi$ of $L$ such that the variation field $V=\Phi_{\ast}(\frac{d}{dt})$ vanishes in some neighborhood of $\partial L$. Clearly, if $L$ is closed, there is no restriction on $V$.

In this section we will derive the formula for the variation field for $u$--minimality. We need to repeat, with some modifications, the bundle approach from the previous section. 

Let $\Phi:L\times\mathbb{R}\to M$ be a mentioned variation. Put $L_t=\Phi_t(L)$.  Denote by $\nabla$ the Levi--Civita connection on $M$. Consider now, the pull--back bundle $\Phi^{-1}TM$ over $L\times\mathbb{R}$, i.e. the fiber of this bundle is of the form $(\Phi^{-1}TM)_{(p,t)}=T_{\Phi(p,t)}M$. The decomposition
\begin{equation}\label{eq:TMsplitting2}
T_{\Phi(p,t)}M=T_{\Phi(p,t)}L_t\oplus T^{\bot}_{\Phi(p,t)}L_t,
\end{equation}
defines the splitting of $\Phi^{-1}TM$ into two bundles denoted by $\Phi^{-1}TL$ and $\Phi^{-1}T^{\bot}L$, respectively. Notice that $\Phi^{-1}TL$ and $\Phi^{-1}T^{\bot}L$ are not pull--back bundles. There is the unique connection $\nabla^{\Phi}$ in $\Phi^{-1}TM$ such that
\begin{equation*}
\nabla^{\Phi}_Z (Y\circ\Phi)=\nabla_{\Phi_{\ast}Z}Y,\quad Z\in T(L\times\mathbb{R}),\quad Y\in\Gamma(TM).
\end{equation*}
The splitting \eqref{eq:TMsplitting2} and connection $\nabla^{\Phi}$ induce connections $\nabla^{\Phi,\top}$ and $\nabla^{\Phi,\bot}$ in $\Phi^{-1}TL$ and $\Phi^{-1}T^{\bot}L$, respectively.

\subsection{Variation of generalized extrinsic curvatures}
Now we are ready to derive the formula for the variation of the generalized extrinsic curvatures. First, we need the formula for the derivative of coefficients of the shape operator. Let $V\in\Gamma(\Phi^{-1}TM)$ be the variation field, i.e. $V=\Phi_{\ast}\frac{d}{dt}$. Decompose $V$ into two components $V^{\top}\in\Gamma(\Phi^{-1}T^{\top}L)$ and $V^{\bot}\in\Gamma(\Phi^{-1}T^{\bot}L)$. 

Let $(e_{\alpha})$ and $(e_i)$ be sections of $\Phi^{-1}T^{\bot}L$ and $\Phi^{-1}TL$, respectively, of orthonormal vector fields. Put
\begin{equation}\label{eq:halphaij}
A^{e_{\alpha}}(e_i)=\sum_j h^{\alpha}_{ij}e_j.
\end{equation}
Then $h^{\alpha}_{ij}$ are functions on $L\times\mathbb{R}$. We have the following formula, which has been obtained by several authors \cite{Li,CL}.

\begin{prop}\label{prop:variationofA}
Fix a point $p\in L$. Choose an orthonormal basis $(e_{\alpha})$ of $T_p^{\bot}L$ and an orthonormal basis $(e_i)$ of $T_pL$. Extend these bases locally on $L\times\mathbb{R}$ such that at $(p,0)$
\begin{equation}\label{eq:suitablebases}
\nabla^{\Phi,\bot}_{\frac{d}{dt}}e_{\alpha}=0\quad\textrm{and}\quad\nabla^{\Phi,\top}_{\frac{d}{dt}}e_i=0\quad\textrm{for all $\alpha$ and $i$}.
\end{equation} 
Then at $(p,0)$
\begin{equation}\label{eq:variationofA}
\frac{d}{dt}h^{\alpha}_{ij}=-\sum_{\beta}(R_{\alpha\beta})_{ij}V_{\beta}+((\nabla^{\bot})^2_{e_i,e_j}V^{\bot})_{\alpha}+(A^{V^{\bot}}A_{\alpha})_{ij}+\left(\left(\nabla^{\top}_{V^{\top}}A\right)_{\alpha}\right)_{ij},
\end{equation}
where $X_{\alpha}$ is the $\alpha$--th coordinate of $X\in T^{\bot}_pL$ with respect to $(e_{\alpha})$ and $B_{ij}$ is the $ij$--th coefficient of the operator $B\in{\rm End}(T_pL)$ with respect to $(e_i)$ and $R_{\alpha\beta}$ is defined in \eqref{eq:Ralphabeta}.
\end{prop}

\begin{rem}
Notice, that in \eqref{eq:variationofA} not every term is symmetric in $i,j$. In fact, the last two terms and the sum of the first and second are.
\end{rem}

Now we pass to the variation of $\hat{\sigma}_u$. Let $P$ be the bundle over $L\times\mathbb{R}$ of orthonormal frames of $\Phi^{-1}T^{\bot}L$. Each element of $P$ is the triple $(p,t,e)$, where $(p,t)\in L\times \mathbb{R}$ and $e$ is an orthonormal basis of $ \Phi^{-1}T^{\bot}L$. Moreover, let $E$ be the bundle over $P$ with the fiber $E_{(p,t,e)}=(\Phi^{-1}TL)_{(p,t)}$. Notice that $E$ is the pull--back bundle with respect to $\pi_P$ of the bundle $\Phi^{-1}TL$. We define, analogously as in the previous sections, the bundle $E'$ over $P$ with the fibers of $\Phi^{-1}T^{\bot}L$. As vector spaces
\begin{equation*}
E_{(p,t,e)}\oplus E'_{(p,t,e)}=(\Phi^{-1}TM)_{(p,t,e)}=T_pM.
\end{equation*}
The shape operator of $L$ induces the family of operators $A_{\alpha}$ as follows
\begin{equation*}
A_{\alpha}(p,t,e)=A^{e_\alpha}_{(p,t)}.
\end{equation*} 
Thus $A_{\alpha}\in \Gamma({\rm End}(E))$. Moreover, the symmetric function $\sigma_u$ of the system ${\bf A}=(A_1,\ldots,A_q)$ is a function on $P$,
\begin{equation*}
\sigma_u(p,t,e)=\sigma_u(A^{e_1}_{(p,t)},\ldots, A^{e_q}_{(p,t)})
\end{equation*} 
and the generalized Newton transformation $T_u$ is the section of ${\rm End}(E)$. 

The following Proposition is crucial the study of $u$--minimality.

\begin{prop}\label{prop:variationsigmau}
The horizontal lift of $\frac{d}{dt}$ to the bundle $P$ evaluated on a generalized symmetric function $\sigma_u$ at $t=0$ equals
\begin{equation}\label{eq:variationsigmau}
\left(\frac{d}{dt}\right)^h\sigma_u=g({\bf R}_u-{\bf S}_u,V)+\sigma_ug(H_L,V)\circ\pi_P
+(V^{\top})^h\sigma_u+L^{\ast}_u(V^{\bot}\circ\pi_P),
\end{equation}
where $H_L$ is the mean curvature of $L$.
\end{prop}
\begin{proof} 
By \eqref{eq:differentialsigmau} we have
\begin{equation*}
\left(\frac{d}{dt}\right)^h\sigma_u=\sum_{\alpha}{\rm tr}
\left(\left(\nabla^E_{\left(\frac{d}{dt}\right)^h}A\right)_{\alpha}\cdot T_{\alpha^{\flat}(u)}\right).
\end{equation*}
Fix a point $p\in L$. Choose an orthonormal basis $(e_i)$ of $T_pL$ and extend locally such that $\nabla^{\Phi,\top}e_i=0$ for all $i$ at $p$. Fix a basis $(p,e)\in P$. We wish to compute $\frac{d}{dt}\sigma_u(\gamma^h_{(p,e)}(t))$, where $\gamma^h_{(p,e)}$ is a horizontal lift of the integral curve $\gamma$ of $\frac{d}{dt}$. Thus along $\gamma$, $\nabla^{\Phi,\bot}_{\frac{d}{dt}}e_{\alpha}=0$ for all $\alpha$. Therefore assumption \eqref{eq:suitablebases} of Proposition \ref{prop:variationofA} is satisfied. Then (see subsection \ref{subsec:bundlecalculus})
\begin{equation*}
\left(\left(\nabla^E_{\left(\frac{d}{dt}\right)^h}A\right)_{\alpha}\right)_{ij}=\frac{d}{dt}h^{\alpha}_{ij}.
\end{equation*}
Thus by Proposition \ref{prop:variationofA} and \eqref{eq:Hessalpha} we get
\begin{align*}
\left(\frac{d}{dt}\right)^h \sigma_u &
=\sum_{\alpha,i,j}\frac{d}{dt}h^{\alpha}_{ij}\cdot(T_{\alpha_\flat(u)})_{ji}\\
&=-\sum_{\alpha,\beta}{\rm tr}(R_{\alpha\beta}T_{\alpha_{\flat}(u)})V_{\beta}+L_u^{\ast}(V^{\bot}\circ\pi_P)\\
&+\sum_{\alpha}{\rm tr}(A^{V_{\bot}}A_{\alpha}T_{\alpha_{\flat}(u)})+\sum_{\alpha}{\rm tr}(\nabla^{\top}_{V^{\top}}A_{\alpha}\cdot T_{\alpha_{\flat}(u)}).
\end{align*}
Since (see subsection \ref{subsec:bundlecalculus} and \eqref{eq:differentialsigmau})
\begin{equation*}
\sum_{\alpha}{\rm tr}\left((\nabla_{V^{\top}}A)^{e_{\alpha}}\cdot T_{\alpha_{\flat}(u)}\right)
=\sum_{\alpha}{\rm tr}\left( \left(\nabla^E_{(V^{\top})^h}A\right)_{\alpha}\cdot T_{\alpha_{\flat}(u)}\right)=(V^{\top})^h\sigma_u
\end{equation*}
and by Proposition \ref{prop:propTu},
\begin{equation*}
\sum_{\alpha}{\rm tr}(A^{V^{\bot}}A_{\alpha}T_{\alpha_{\flat}(u)})
={\rm tr}(A^{V^{\bot}}(\sigma_uI-T_u))=\sigma_u g(H_L,V)-{\rm tr}(A^{V^{\bot}}T_u),
\end{equation*}
we obtain \eqref{eq:variationsigmau}.
\end{proof}

Let us turn to the main considerations of the article, namely, let us compute the variation of $\sigma_u$. Put $P_t=\pi_{P}^{-1}(\cdot,t)$. Then, differentiating under the integral sign (compare subsection \ref{subsec:bundlecalculus}) at $t=0$
\begin{align*}
\frac{d}{dt}\int_{P_t}\sigma_u\,d{\rm vol}_{P_t}
&=\frac{d}{dt}\int_{L_t}\int_{P_{(p,t)}}\sigma_u\,d{\rm vol}_{P_{(p,t)}}\,d{\rm vol}_{L_t}\\
&=\int_L\int_{P_{(p,0)}}\left(\frac{d}{dt}\right)^h\sigma_u\,d{\rm vol}_{P_{(p,0)}}\,d{\rm vol}_L\\
&+\int_L\int_{P_{(p,0)}}\sigma_u\,d{\rm vol}_{P_{(p,0)}}\,d\frac{d}{dt}{\rm vol}_{L_t}. 
\end{align*}
Since
\begin{equation*}
\frac{d}{dt}{\rm vol}_{L_t}=({\rm div}V^{\top}-g(H_L,V)){\rm vol}_L,
\end{equation*}
by the fact that $V^{\top}$ vanishes in some neighborhood of $\partial L$, for any smooth function $f$ we have
\begin{equation*}
\int_L f\,d\frac{d}{dt}{\rm vol}_{L_t}=\int_L (f{\rm div}V^{\top}-fg(H_L,V))\,d{\rm vol}_L=\int_L (-V^{\top}f-fg(H_L,V))\,d{\rm vol}_L.
\end{equation*}
Thus by Proposition \ref{prop:variationsigmau}, \eqref{eq:Luastintegral} and again subsection \ref{subsec:bundlecalculus}
\begin{align*}
\frac{d}{dt}\int_{P_t}\sigma_u\,d{\rm vol}_{P_t} &
=\int_L\int_{P_{(p,0)}} (g({\bf R}_u-{\bf S_u},V)\,d{\rm vol}_{P_{(p,0)}}\,d{\rm vol}_L \\
&+\int_L\int_{P_{(p,0)}}\sigma_ug(H_L,V)\circ\pi_P\,d{\rm vol}_{P_{(p,0)}}\,d{\rm vol}_L\\
&+\int_L\int_{P_{(p,0)}}(V^{\top})^h\sigma_u\,d{\rm vol}_{P_{(p,0)}}\,d{\rm vol}_L+\int_{P_0} L^{\ast}_u(V^{\bot}\circ\pi_P)\,d{\rm vol}_{P_0}\\
&-\int_L V^{\top}\int_{P_{(p,0)}}\sigma_u\,d{\rm vol}_{P_{(p,0)}}\,d{\rm vol}_L-\int_L g(H_L,V)\int_{P_{(p,0)}}\sigma_u\,d{\rm vol}_{P_{(p,0)}}\,d{\rm vol}_L\\
&=\int_L\int_{P_{(p,0)}} g({\bf R}_u-{\bf S}_u,V)\,d{\rm vol}_{P_{(p,0)}}\,d{\rm vol}_L\\
&+\int_{P_0}\sum_{\alpha}V_{\alpha}{\rm div}_E({\rm div}_E T_{\alpha_{\flat}(u)})\,d{\rm vol}_{P_0}\\
&=\int_L\int_{P_{(p,0)}}g({\bf R}_u-{\bf S}_u+{\bf W}_u,V)\,d{\rm vol}_{P_{(p,0)}}\,d{\rm vol}_L.
\end{align*}
By the definitions of $R_u$, $S_u$, $W_u$ and since $P_0$ is just the bundle $P$ we finally get
\begin{equation}\label{eq:sigmauvariation}
\frac{d}{dt}\left(\int_{P_t}\sigma_u\,d{\rm vol}_{P_t}\right)_{t=0}
=\int_L g(R_u-S_u+W_u,V)\,d{\rm vol}_L.
\end{equation}

\begin{thm}\label{thm:uminimality}
A submanifold $L$ is $u$--minimal if and only if
\begin{equation*}
R_u+W_u=S_u.
\end{equation*}
In particular, assuming $M$ is of constant sectional curvature $c$, then $L$ is $u$--minimal if and only if
\begin{equation}\label{eq:uminimalityconstant}
c(n+1-|u|)H_u=S_u,\quad n=\dim L.
\end{equation}
\end{thm}
\begin{proof}
The first part follows directly by above consideration. For the second part, it is easy to see by \eqref{eq:divTu} that ${\rm div}_E T_v=0$ for all choices of $v$. Thus, by the first part, $L$  is $u$--minimal if and only if $R_u=S_u$. It suffices to notice that for $M$ of constant sectional curvature $c$, by Proposition \ref{prop:propTu},
\begin{equation*}
{\bf R}_u=\sum_{\alpha,\beta}{\rm tr}( c\delta_{\alpha\beta}I\cdot T_{\alpha_{\flat}(u)})e_{\beta}=c\sum_{\alpha}{\rm tr}(T_{\alpha_{\flat}(u)})e_{\alpha}=c(n+1-|u|){\bf H}_u.\qedhere
\end{equation*}
\end{proof}

\section{Examples of $u$--minimal submanifolds}\label{sec:Examples}

In this section we give examples of $u$--minimal submanifolds in the space forms for certain choices of multi--index $u$. Adopt the notation from the Section \ref{sec:Submgeom}.

\begin{exa}\label{ex:1}
Totally geodesic submanifolds are clearly $u$--minimal for all possible choices of $u\in\mathbb{N}^q$.
\end{exa}

\begin{exa}\label{ex:2}
Let $u=0=(0,\ldots,0)$. Then $\alpha_{\flat}(u)=0$, so $W_u$ vanishes. Further, we have
\begin{equation*}
{\bf S}_0=\sum_{\alpha}{\rm tr}(A_{\alpha})e_{\alpha}=H\circ \pi_P
\end{equation*}
and, clearly ${\bf R}_u=0$. Thus $0$--minimality is equivalent to classical notion of minimality of a submanifold.
\end{exa}

\begin{exa}\label{ex:3}
Let $L$ be a totally umbilical codimension $q$ submanifold in the space form $M^{n+q}$ with sectional curvature $c$. Then
\begin{equation*}
A_{\alpha}=\frac{1}{n}\lambda_{\alpha}{\rm Id},\quad \lambda_{\alpha}=H_{\alpha}=g(H,e_{\alpha}),
\end{equation*}
where $H$ is a mean curvature of $L$. By the multinomial theorem the characteristic polynomial $\chi_{\bf A}$ equals
\begin{equation*}
\chi_{\bf A}({\bf t})=\left(1+t_1\frac{\lambda_1}{n}+\ldots+t_q\frac{\lambda_q}{n}\right)^n
=\sum_{|u|\leq n}\frac{n!}{n^{|u|}(n-|u|)!u!}{\bf t}^u\pmb{\lambda}^{u},
\end{equation*}
where $\pmb{\lambda}=(\lambda_1,\ldots,\lambda_q)$. Hence
\begin{equation*}
\sigma_u=\frac{n!}{n^{|u|}(n-|u|)!u!}\pmb{\lambda}^u.
\end{equation*}
Therefore
\begin{equation*}
{\bf H}_u=\sum_{\alpha}\sigma_{\alpha_{\flat}(u)}e_{\alpha}
=\frac{n!}{n^{|u|-1}(n+1-|u|)!u!}\sum_{\alpha}u_{\alpha}\pmb{\lambda}^{\alpha_{\flat}(u)}e_{\alpha}
\end{equation*}
and, by Proposition \ref{prop:propTu},
\begin{equation*}
{\bf S}_u=\sum_{\alpha}{\rm tr}(A_{\alpha}T_u)e_{\alpha}
=\frac{n-|u|}{n}\sigma_u\sum_{\alpha}\lambda_{\alpha}e_{\alpha}
=\frac{n!}{n^{|u|+1}(n-1-|u|)!u!}\pmb{\lambda}^uH\circ\pi_P.
\end{equation*}
Thus by Theorem \ref{thm:uminimality}, $L$ is $u$--minimal if and only if for any $p\in L$
\begin{equation}\label{eq:exumbilical}
cn^2\int_{P_p}\sum_{\alpha}\pmb{\lambda}^{\alpha_{\flat}(u)}u_{\alpha}e_{\alpha}\,d{\rm vol}_{P_p}=(n-|u|)(n+1-|u|)\left( \int_{P_p}\pmb{\lambda}^u\,d{\rm vol}_{P_p}\right) H(p). 
\end{equation}
Evaluating the inner product of both sides with $H_L$, \eqref{eq:exumbilical} implies
\begin{equation}\label{eq:exumbilical2}
(cn^2|u|-(n-|u|)(n+1-|u|)|H(p)|^2)\left(\int_{P_p}\pmb{\lambda}^u\,d{\rm vol}_{P_p}\right)=0.
\end{equation} 
Hence, if $\int_{P_p}\pmb{\lambda}^u\,d{\rm vol}_{P_p}$ is not identically equal to zero and $H$ is not of  constant length, then $L$ is not $u$--minimal for any choice of $u$.

Taking the inner product of both sides of \eqref{eq:exumbilical} with a vector field $X\in \Gamma(T^{\bot}L)$ orthogonal to $H$ we get
\begin{equation}\label{eq:exumbilical3}
c\int_{P_p}\sum_{\alpha}\pmb{\lambda}^{\alpha_{\flat}(u)}u_{\alpha}\mu_{\alpha}\,d{\rm vol}_{P_p}=0,
\end{equation}
where $X=\sum_{\alpha}\mu_{\alpha}e_{\alpha}$.

Assume now $u_{\alpha}=k$ for all $\alpha$, where $k>1$ is even. Then, it can be shown that the function under the integral sign is constant, i.e. does not depend on the choice of the basis $(p,e)$. Thus we may evaluate this function on the basis $(p,e)\in P_p$ such that $e_1=\frac{1}{|X|}X$ and $e_2=\frac{1}{|H|}H$ (assuming $H_L\neq 0$). Then only nonzero elements in $\pmb{\lambda}$ and $\pmb{\mu}$ are $\lambda_2=|H|$ and $\mu_1=|X|$. Hence
\begin{equation*}
\sum_{\alpha}\pmb{\lambda}^{\alpha_{\flat}(u)}u_{\alpha}\mu_{\alpha}
=k\sum_{\alpha}\lambda_1^k\ldots \lambda_{\alpha}^{k-1}\ldots\lambda_q^k\mu_{\alpha}
=0.
\end{equation*}
Thus \eqref{eq:exumbilical3} holds. Concluding, we have the following corollary.
\begin{cor}\label{cor:umbilical}
Let $L$ be a totally umbilical codimension $q$ submanifold in the space form $M^{n+q}$ with sectional curvature $c$. If one of the following conditions hold
\begin{enumerate}
\item $c=0$ and $|u|=n$,
\item $u=(k,k,\ldots,k)$, where $k>1$ is even, and $(n-qk)(n+1-qk)|H|^2=cqkn^2$,
\end{enumerate} 
then $L$ is $u$--minimal.
\end{cor}
\end{exa}

\begin{exa}\label{ex:4}
Let us derive the formula for the $u$--minimality in the first nontrivial case, namely, for $u=\beta^{\sharp}\beta^{\sharp}(0)$, where $\beta^{\sharp}(v)$ denotes the multi--index obtained from $v$ by adding $1$ on $\alpha$--th coordinate. Assume that $M$ is of constant sectional curvature $c$. By the properties of the generalized Newton transformation we have
\begin{equation*}
T_u=\sigma_u I-\sigma_{\beta^{\sharp}(0)}A_{\beta}+A_{\beta}^2,
\end{equation*}
where
\begin{equation*}
\sigma_u=\frac{1}{2}\left(g(H,e_{\beta})^2-|A_{\beta}|^2\right),\quad
\sigma_{\beta^{\sharp}(0)}=g(H,e_{\beta}).
\end{equation*}
Therefore, after some computations,
\begin{align*}
{\bf H}_u &=g(H,e_{\beta})e_{\beta},\\
{\bf S}_u &=\frac{1}{2}\left(g(H,e_{\beta})^2-|A_{\beta}|^2\right)H\circ\pi_P-
g(H,e_{\beta})\sum_{\alpha}{\rm tr}(A_{\alpha}A_{\beta})e_{\alpha}+
\sum_{\alpha}{\rm tr}(A_{\alpha}A_{\beta}^2)e_{\alpha}.
\end{align*}
Notice that the integrals of ${\bf H}_u$ and ${\bf S}_u$ do not depend on the choice of $\beta$. Therefore, we may consider the following sections
\begin{equation*}
\sum_{\beta}{\bf H}_{\beta^{\sharp}\beta^{\sharp}(0)}=H\circ \pi_P
\end{equation*} 
and
\begin{equation*}
\sum_{\beta} {\bf S}_{\beta^{\sharp}\beta^{\sharp}(0)}=\left(
\frac{1}{2}(|H|^2-|B|^2)H-{\rm tr}(B\circ A^H)+{\rm tr}(B\circ A^2)\right)\circ\pi_P,
\end{equation*}
where $A^2=\sum_{\beta}A_{\beta}^2$ and the composition means substituting on one of the factors, i.e., $B\circ A^H=B(A^H,\cdot)=B(\cdot, A^H)$, etc. Thus the condition for $u$--minimality is equivalent to the following equation
\begin{equation*}
\left(\frac{1}{2}(|H|^2-|B|^2)-c(n-1)\right)H
-{\rm tr}(B\circ A^H)+{\rm tr}(B\circ A^2)=0.
\end{equation*}
Assuming, additionally, that $L$ is minimal the above formula simplifies to
\begin{equation*}
{\rm tr}(B\circ A^2)=0.
\end{equation*}

Consider the second standard immersion $\varphi:S^2(1)\mapsto S^4(\frac{1}{\sqrt{3}})$ \cite{Che}. Then $S^2(1)$ is minimal but not totally geodesic. It was shown in \cite{KN} that the second fundamental form $B$ of this immersion satisfies the following relation
\begin{equation*}
g(B(X,Z),B(Y,Z))=g(X,Y)g(Z,Z),\quad X,Y,Z\in TS^2(1).
\end{equation*}
Notice that the operator $A^2$, by the above formula, equals
\begin{equation*}
A^2X=\sum_{j,k}g(B(X,e_k),B(e_j,e_k))e_j=2\sum_jg(X,e_j)e_j=2X.
\end{equation*}
Thus $A^2=2 I$ and, therefore, ${\rm tr}(B\circ A^2)=2{\rm tr}B=2H=0$.  Finally, the second standard immersion is $\beta^{\sharp}\beta^{\sharp}(0)$--minimal for any choice of $\beta$. 
\end{exa}


\begin{thebibliography}{99}
\bibitem{AC} L. J. Alias, A. G. Colares, Uniqueness of spacelike hypersurfaces with constant higher order mean curvature in generalized Robertson--Walker spacetimes, Math. Proc. Camb. Phil. Soc. 143 (2007), 703--729.
\bibitem{ALM} L. J. Alias, S. de Lira, J. M. Malacarne, Constant higher-order mean curvature hypersurfaces in Riemannian spaces, J Inst. of Math. Jussieu 5(4) (2006), 527--562.
\bibitem{AKN} K. Andrzejewski, W. Koz\l owski, K. Niedzia\l omski, Generalized Newton transformation and its applications to extrinsic geometry, Asian J. Math., to appear, Arxiv, http://arxiv.org/abs/1211.4754
\bibitem{AW1} K. Andrzejewski, P. Walczak, Extrinsic curvatures of distributions of arbitrary codimension, J. Geom. Phys. 60 (2010), no. 5, 708--713.
\bibitem{AW2} K. Andrzejewski, P. Walczak, Conformal fields and the stability of leaves with constant higher order mean curvature, Differential Geom. Appl. 29 (2011), no. 6, 723--729.
\bibitem{BC} J. L. M. Barbosa, A. G. Colares, Stability of hypersurfaces with constant $r$--mean curvature, Ann. Global Anal. Geom. 15 (1997), 277--297.
\bibitem{BS} A. Barros, P. Sousa, Compact graphs over a sphere of constant second order mean curvature. Proc. Am. Math. Soc. 137(9) (2009), 3105--3114.
\bibitem{BN} F. Brito, A. M. Naveira, Total extrinsic curvature of certain distributions on closed spaces of constant curvature, Ann. Global Anal. Geom. 18 (2000), 371--383.
\bibitem{CL} L. Cao, H. Li, $r$--Minimal submanifolds in space forms, Ann. Global Anal. Geom. 32 (2007), 311--341.
\bibitem{CL2} L. Cao, H. Li, Variational Problems in Geometry of Submanifolds,  Variational problems in geometry of submanifolds. Proceedings of the Eleventh International Workshop on Differential Geometry, 41--71, Kyungpook Nat. Univ., Taegu (2007).
\bibitem{Che} B. Chen, Total mean curvature and submanifolds of finite type. World Scientific, Singapore (1984).
\bibitem{CR} X. Cheng, H. Rosenberg, Embedded positive constant $r$--mean curvature hypersurfaces in $M^m\times\mathbb{R}$, An. Acad. Brasil. Cienc.  77 (2005), no. 2, 183--199.
\bibitem{Gro} J. Grosjean, Upper bounds for the first eigenvalue of the Laplacian on compact submanifolds, Pacific J. Math. 206 (2002), 93--112.
\bibitem{KN} W. Koz\l owski, K. Niedzialomski, Conformality of a differential with respect to Cheeger-Gromoll type metrics, Geom. Dedicata 157 (2012), 227--237.
\bibitem{Li} H. Li, Hypersurfaces with constant scalar curvature in space forms, Math. Ann., 305 (1996), 665--672.
\bibitem{Rei} R. Reilly, Variational properties of functions of the mean curvatures for hypersurfaces in space forms, J. Differential Geom. 8 (1973), 465--477.
\bibitem{Rei2} R. Reilly, On the first eigenvalue the Laplacian for compact submanifolds of Euclidean space, Comment. Math. Helvetici. 52 (1977), 465--477.
\bibitem{Ros} H. Rosenberg, Hypersurfaces of Constant Curvature in Space of Forms, Bull. Sci. Math. (1993), Vol 117, 211--239.
\bibitem{Rov} V. Rovenski, Integral formulae for a Riemannian manifold with two orthogonal distributions. Cent. Eur. J. Math. 9 (2011), no. 3, 558--577.
\bibitem{RW} V. Rovenski, P. Walczak,  Integral formulae on foliated symmetric spaces. Math. Ann. 352 (2012), no. 1, 223--237.
\bibitem{Sim} J. Simons, Minimal varieties in Riemmannian manifolds. Ann. Math. 2nd Ser. 88(1) (1968), 62--105.
\end{thebibliography}
\end{document}